\documentclass[a4paper]{amsart}
\usepackage{graphicx,amsmath,amsfonts,latexsym,amssymb,amsthm,mathrsfs,color}
\usepackage[latin1]{inputenc}
\newtheorem{thm}{Theorem}[section]
\newtheorem{proposition}[thm]{Proposition}

\newtheorem{definition}[thm]{Definition}
\newtheorem{theorem}[thm]{Theorem}
\newtheorem{lemma}[thm]{Lemma}
\newtheorem{remark}[thm]{Remark}
\newtheorem{notation}[thm]{Notation}
\newtheorem{condition}[thm]{Condition}

\begin{document}
\title[]%
{Closedness and invertibility for the sum of two closed operators}

\author{Nikolaos Roidos}
\address{Institut f\"ur Analysis, Leibniz Universit\"at Hannover, Welfengarten 1, 30167 Hannover, Germany}
\email{roidos@math.uni-hannover.de}

\begin{abstract}
We show a Kalton-Weis type theorem for the general case of non-commuting operators. More precisely, we consider sums of two possibly non-commuting linear operators defined in a Banach space such that one of the operators admits a bounded $H^\infty$-calculus, the resolvent of the other one satisfies some weaker boundedness condition and the commutator of their resolvents has certain decay behavior with respect to the spectral parameters. Under this consideration, we show that the sum is closed and that after a sufficiently large positive shift it becomes invertible, and moreover sectorial. As an application we recover a classical result on the existence, uniqueness and maximal $L^{p}$-regularity for solutions of the abstract linear non-autonomous parabolic problem.
\end{abstract}

\date{\today}

\maketitle

\section{Introduction}

Let $E$ be a complex Banach space and $A:\mathcal{D}(A)\rightarrow E$, $B:\mathcal{D}(B)\rightarrow E$ be two closed possibly non-commuting linear operators in $E$. We consider the question of whether the sum $A+B$ with domain $\mathcal{D}(A+B)=\mathcal{D}(A)\cap\mathcal{D}(B)$ is also closed. Furthermore, we ask under which assumptions the sum becomes invertible. The last is related to the existence, uniqueness and maximal regularity for solutions of the following abstract linear equation, namely 
$$
(A+B)x=y, \quad y\in E.
$$ 
We distinguish between two cases according to whether the operators commute or not, where by commuting we mean resolvent commuting (see e.g. \cite[(III.4.9.1)]{Am}).

For a quick review on the above two problems, we start with the classical result of Da Prato and Grisvard \cite{DG}. They showed that the sum of two sectorial operators is closable and the closure is invertible if the strong parabolicity condition on their sectoriality angles is fulfilled. In the non-commuting case they required the commutation condition \cite[(6.5)]{DG}.

If we restrict to the commuting case, then we mention two classical results that further provide closedness. By employing the underlying properties of UMD Banach spaces, Dore and Venni in \cite{DV} showed that the sum of two operators that have bounded imaginary powers is closed and invertible, provided that the strong parabolicity condition on their power angles is satisfied. Next, Kalton and Weis in \cite{KW1} treated the problem as a special case of operator valued holomorphic functional calculus. Without assumption on the geometry of the Banach space, they showed that if one of the operators admits a bounded $H^\infty$-calculus, the other one is $R$-sectorial (or even $U$-sectorial) and the strong parabolicity condition on the corresponding angles is fulfilled, then the sum is closed and invertible. 

For the general case of non-commuting operators we mention the following two remarkable generalizations. First, Monniaux and Pr\"uss in \cite{MP} showed that the Dore-Venni theorem can be extended to the non-commuting case if furthermore the Labbas-Terreni condition \cite[(2.6)]{MP} is satisfied. Then, Pr\"uss and Simonett, extended the Kalton-Weis theorem to the general case of non-commuting operators provided that either the Da Prato-Grisvard condition \cite[(3.1)]{PS} or the Labbas-Terreni condition \cite[(3.2)]{PS} is satisfied.

In the present paper we give an answer to the problems of closedness and invertibility of the sum of two operators simultaneously (Theorem \ref{t2}). We extend the Kalton-Weis theorem to the case of possibly non-commuting operators under certain decay condition on the commutator of the two resolvents with respect to the spectral parameters (Condition \ref{c3}). Moreover, we find one more commutation condition that is stronger than the above one but it does not imply the Da Prato-Grisvard condition or the Labbas-Terreni condition (see Remark \ref{rcomp}). In addition, instead of asking for one of the summands to be $R$-sectorial we introduce a weaker boundedness condition, which is a boundedness property for the resolvent based on Bochner-norm estimates over an arbitrary measure space (Definition \ref{don}). 

 As an application of our main theorem, we recover a well-known result on the existence, uniqueness and maximal $L^{p}$-regularity for solutions of the abstract linear non-autonomous parabolic problem in UMD spaces (Theorem \ref{tf}). We treat the problem by constructing the inverse of the sum of the time derivative and the non-automomous term. The benefit of our approach is that the above inverse is explicitly expressed by using Neumann series in terms of the pointwise freezings of the non-autonomous operator. As a consequence, the mapping properties of the inverse can be easily controlled in terms of the data in our construction. Therefore, maximal $L^p$-regularity space estimates for the solutions can be obtained, which are useful e.g. for showing existence of long time solutions for quasilinear parabolic problems (see e.g. \cite[Theorem 4.6]{RS}).

\section{Notation and preliminaries}

Denote by $\rho$ and $\sigma$ the resolvent and the spectrum of a linear operator respectively. We start with the basic notion of a sectorial operator.

\begin{definition}
Let $E$ be a complex Banach space and $\theta\in[0,\pi)$. Let $\mathcal{P}_{\kappa}(\theta)$, $\kappa\geq1$, be the class of all closed densely defined linear operators $A$ in $E$ such that 
$$
S_{\theta}=\{\lambda\in\mathbb{C}\,|\, |\arg \lambda|\leq\theta\}\cup\{0\}\subset\rho{(-A)} \quad \mbox{and} \quad (1+|\lambda|)\|(A+\lambda)^{-1}\|_{\mathcal{L}(E)}\leq \kappa, \quad \forall \lambda\in S_{\theta}.
$$
The elements in $\mathcal{P}(\theta)=\cup_{\kappa\geq1}\mathcal{P}_{\kappa}(\theta)$ are called {\em (invertible) sectorial operators of angle $\theta$}. If $A\in \mathcal{P}(\theta)$ then any $\kappa$ such that $A\in \mathcal{P}_{\kappa}(\theta)$ is called {\em sectorial bound of $A$} and the constant $\inf\{\kappa\, |\, A\in \mathcal{P}_{\kappa}(\theta)\}$ depends on $A$ and $\theta$.
\end{definition}

If $A\in \mathcal{P}_{\kappa}(\theta)$, then a sectoriality area extension argument (see e.g. \cite[(III.4.7.11)]{Am} or the Appendix of \cite{Ro1}) implies that 
\begin{gather*}
\Omega_{\kappa,\theta}=\bigcup_{z\in S_{\theta}}\Big\{\lambda\in\mathbb{C}\,|\,|\lambda-z|\leq\frac{1+|z|}{2\kappa}\Big\}\subset \rho{(-A)} 
\end{gather*}
and
\begin{gather*}
(1+|\lambda|)\|(A+\lambda)^{-1}\|_{\mathcal{L}(E)}\leq 2\kappa+1, \quad \forall \lambda\in \Omega_{\kappa,\theta}.
\end{gather*}
Therefore, whenever $A\in \mathcal{P}(\theta)$ we can assume that $\theta>0$ (see e.g. \cite[(III.4.6.4)]{Am} and \cite[(III.4.6.5)]{Am}).
For any $\rho\geq0$ and $\theta\in(0,\pi)$, let the positively oriented path 
\[
\Gamma_{\rho,\theta}=\{re^{-i\theta}\in\mathbb{C}\,|\,r\geq\rho\}\cup\{\rho e^{i\phi}\in\mathbb{C}\,|\,\theta\leq\phi\leq2\pi-\theta\}\cup\{re^{+i\theta}\in\mathbb{C}\,|\,r\geq\rho\},
\]
where we denote $\Gamma_{0,\theta}$ simply by $\Gamma_{\theta}$. We can define holomorphic functional calculus for sectorial operators by using the Dunford integral formula. Then, the following basic property can be satisfied. 

\begin{definition} 
Let $E$ be a complex Banach space, $\theta\in(0,\pi)$, $\phi\in[0,\theta)$ and $A\in\mathcal{P}(\theta)$. Let $H_{0}^{\infty}(\phi)$ be the space of all bounded holomorphic functions $f:\mathbb{C}\backslash S_{\phi}\rightarrow \mathbb{C}$ such that 
\[
|f(\lambda)|\leq c \Big(\frac{|\lambda|}{1+|\lambda|^{2}}\Big)^{\eta} \quad \mbox{for any} \quad \lambda\in \mathbb{C}\backslash S_{\phi}, 
\]
with some $c>0$ and $\eta>0$ depending on $f$. Any $f\in H_{0}^{\infty}(\phi)$ defines an element $f(-A)\in \mathcal{L}(E)$ by 
\[
f(-A)=\frac{1}{2\pi i}\int_{\Gamma_{\theta}}f(\lambda)(A+\lambda)^{-1} d\lambda.
\]
We say that the operator $A$ {\em admits a bounded $H^{\infty}$-calculus of angle $\phi$}, and we denote by $A\in \mathcal{H}^{\infty}(\phi)$, if
\begin{gather*}
\|f(-A)\|_{\mathcal{L}(E)}\leq C_{A,\phi}\sup_{\lambda\in\mathbb{C}\backslash S_{\phi}}|f(\lambda)| \quad \mbox{for any} \quad f\in H_{0}^{\infty}(\phi),
\end{gather*}
with some constant $C_{A,\phi}>0$ that is called {\em$H^\infty$-bound} of $A$ and depends only on $A$ and $\phi$.
\end{definition}

Let $A:\mathcal{D}(A)\rightarrow E$ be a linear operator in a complex Banach space $E$ such that $A\in \mathcal{P}(\theta)\cap\mathcal{H}^{\infty}(\phi)$, $0\leq\phi<\theta<\pi$. Denote by $A^{\ast}:\mathcal{D}(A^{\ast})\rightarrow E^{\ast}$ the adjoint of $A$ defined in the continuous dual space $E^{\ast}$ of $E$. Then, $A^{\ast}\in \mathcal{P}(\theta)\cap\mathcal{H}^{\infty}(\phi)$ provided that $\mathcal{D}(A^\ast)$ is dense in $E^{\ast}$, which will be always assumed in the sequel. This is \cite[Proposition 1.3 (v)]{DHP} and \cite[Proposition 2.11 (v)]{DHP}. We recall next a boundedness property of operators having bounded $H^\infty$-functional calculus, which will be of particular importance in our estimates later on. Let $\mathbb{D}=\{z\in\mathbb{C}\, |\, |z|\leq1\}$ be the closed unit disk in $\mathbb{C}$ and $\mathbb{N}=\{1,2,...\}$.

\begin{lemma}\label{LKW}
Let $E$ be a complex Banach space, $A\in\mathcal{H}^{\infty}(\phi)$ and $f\in H_{0}^{\infty}(\phi)$. For any $t>0$ and any finite sequence $\{a_{k}\}_{k\in\{0,...,n\}}$, $n\in\mathbb{N}\cup\{0\}$, with $a_{k}\in \mathbb{D}$ for each $k$, we have that 
$$
\|\sum_{k=0}^{n}a_{k}f(-t2^{-k}A)\|_{\mathcal{L}(E)}\leq C_{A,\phi,f}
$$
for some constant $C_{A,\phi,f}>0$ depending only on the $H^\infty$-bound of $A$, $\phi$ and $f$.
\end{lemma}
\begin{proof}
This is \cite[Lemma 4.1]{KW1}.
\end{proof}

A typical example of the functional calculus for a sectorial operator $A\in\mathcal{P}(\theta)$ are the complex powers. For $\mathrm{Re}(z)<0$ they are defined by
\begin{gather}\label{cp}
A^{z}=\frac{1}{2\pi i}\int_{\Gamma_{\rho,\theta}}(-\lambda)^{z}(A+\lambda)^{-1}d\lambda,
\end{gather}
where $\rho>0$ is sufficiently small. The above family together with $A^{0}=I$ is a strongly continuous holomorphic semigroup on $E$ (see e.g. \cite[Theorem III.4.6.2 ]{Am} and \cite[Theorem III.4.6.5]{Am}). Note that by a sectoriality area extension argument, we can replace $\Gamma_{\rho,\theta}$ in \eqref{cp} by $-\delta+\Gamma_{\theta}$ with $\delta>0$ sufficiently small. Moreover, each operator $A^{z}$, $\mathrm{Re}(z)<0$, is injection and the complex powers for positive real part are defined by $(A^{z})^{-1}$. The imaginary powers are defined as the closure of a variation of formula \eqref{cp}. We refer to \cite[Section III.4.6]{Am} for a detailed description. Next, we recall the following elementary decay property of the resolvent of a sectorial operator.

\begin{lemma}\label{l1}
Let $E$ be a complex Banach space, $\rho\in(0,1)$ and $A\in\mathcal{P}_{\kappa}(\theta)$, $\theta\in(0,\pi)$, $\kappa\geq1$. Then, for any $\phi\in[0,\theta)$ and any $\eta\in[0,1-\rho)$ we have that 
$$
\|A^{\rho}(A+z)^{-1}\|_{\mathcal{L}(E)}\leq\frac{\gamma}{1+|z|^{\eta}}, \quad z\in S_{\phi},
$$
for some constant $\gamma>0$ depending on $\kappa$, $\theta$, $\rho$, $\phi$ and $\eta$.
\end{lemma}
\begin{proof}
For any $z\in S_{\phi}$, by Cauchy's theorem we have that
\begin{eqnarray*}\nonumber
\lefteqn{A^{\rho}(A+z)^{-1}=\frac{1}{2\pi i}A\int_{-\delta+\Gamma_{\theta}}(-\lambda)^{\rho-1}(A+\lambda)^{-1}(A+z)^{-1}d\lambda}\\\nonumber
&=&\frac{1}{2\pi i}A\int_{-\delta+\Gamma_{\theta}}\frac{(-\lambda)^{\rho-1}}{z-\lambda}((A+\lambda)^{-1}-(A+z)^{-1})d\lambda\\\nonumber
&=&\frac{1}{2\pi i}A\int_{-\delta+\Gamma_{\theta}}\frac{(-\lambda)^{\rho-1}}{z-\lambda}(A+\lambda)^{-1}d\lambda-\frac{1}{2\pi i}A(A+z)^{-1}\int_{-\delta+\Gamma_{\theta}}\frac{(-\lambda)^{\rho-1}}{z-\lambda}d\lambda\\\nonumber
&=&\frac{1}{2\pi i}\int_{-\delta+\Gamma_{\theta}}\frac{(-\lambda)^{\rho-1}}{z-\lambda}(A+\lambda-\lambda)(A+\lambda)^{-1}d\lambda\\\nonumber
&=&\frac{1}{2\pi i}\int_{-\delta+\Gamma_{\theta}}\frac{(-\lambda)^{\rho-1}}{z-\lambda}yd\lambda+\frac{1}{2\pi i}\int_{-\delta+\Gamma_{\theta}}\frac{\lambda(-\lambda)^{\rho-1}}{\lambda-z}(A+\lambda)^{-1}d\lambda\\\label{frac}
&=&\frac{1}{2\pi i}\int_{-\delta+\Gamma_{\theta}}\frac{\lambda(-\lambda)^{\rho-1}}{\lambda-z}(A+\lambda)^{-1}d\lambda,
\end{eqnarray*}
for some fixed $\delta>0$ sufficiently small due to a sectoriality area extension argument. Therefore, we obtain
\begin{gather*}
z^{\eta}A^{\rho}(A+z)^{-1}=\frac{1}{2\pi i}\int_{\delta-\Gamma_{\theta}}\frac{(\frac{z}{\lambda})^{\eta}}{1+\frac{z}{\lambda}}\lambda^{\rho+\eta-1}(A-\lambda)^{-1}d\lambda,
\end{gather*}
and the estimate follows by this formula.
\end{proof}

Next we focus on families of bounded operators and introduce a boundedness property with respect to orthonormal sets on an arbitrary measure space.

\begin{notation}
Denote by $\mathcal{S}=(\Omega,\Sigma,\mu)$ an arbitrary finite measure space and by $\mathcal{E}_{n}=\{e_{1},...,e_{n}\}$, $n\in\mathbb{N}$, a finite sequence of vectors in $L^{\infty}(\Omega;\mathbb{C},d\mu)$ with $\|e_{k}\|_{L^{\infty}(\Omega;\mathbb{C},d\mu)}\leq 1$, $k\in\{1,...,n\}$, such that all vectors in $\mathcal{E}_{n}$ are orthonormal in $L^{2}(\Omega;\mathbb{C},d\mu)$. Furthermore, let $E$ be a complex Banach space and denote by $\mathcal{X}_{n}=\{x_{1},...,x_{n}\}$ a finite sequence of vectors in $E$. Also, if $\mathcal{F}\subset\mathcal{L}(E)$ is a family of bounded operators on $E$, denote by $\mathcal{T}_{n}=\{T_{1},...,T_{n}\}$ a finite sequence of vectors in $\mathcal{F}$. Finally, denote by $L^{2}(\Omega;E,d\mu)$ the Bochner space. 
\end{notation}

\begin{definition}\label{don}
Let $E$ be a complex Banach space and $\mathcal{F}\subset\mathcal{L}(E)$ be a family of bounded operators on $E$. According to the previous notation, $\mathcal{F}$ is called {\em orthonormally bounded with respect to the measure space $\mathcal{S}$} if for any triple $\tau=(n, \mathcal{X}_{n}, \mathcal{T}_{n})$ there exists some $\mathcal{E}_{n}$ that depends on $\tau$, such that 
$$
\|\sum_{k=1}^{n}e_{k}T_{k}x_{k}\|_{L^{2}(\Omega;E,d\mu)}\leq C_{\mathcal{F},\mathcal{S}}\big(\sup_{a_{k}\in \mathbb{D}}\|\sum_{k=1}^{n}a_{k}x_{k}\|_{E}\big)
$$
for some constant $C_{\mathcal{F},\mathcal{S}}\geq 1$ which is called {\em orthonormal bound} or {\em $ON$-bound} and depends only on $\mathcal{F}$ and $\mathcal{S}$. If for some family $\mathcal{F}$ there exists some finite measure space $S$ such that $\mathcal{F}$ is orthonormally bounded with respect to $S$, then we say that $\mathcal{F}$ is {\em orthonormally bounded} or {\em $ON$-bounded}.
\end{definition}

In our estimates we will actually require a boundedness condition weaker than the $ON$-boundedness, which is described in the following lemma. 

\begin{lemma}\label{weakon}
Let $E$ be a complex Banach space and $\mathcal{F}\subset\mathcal{L}(E)$ be $ON$-bounded with $ON$-bound equal to $C_{\mathcal{F},\mathcal{S}}$ with respect to some measure space $\mathcal{S}=(\Omega,\Sigma,\mu)$. Then, for any $x_{1},...,x_{n}\in E$, $x_{1}^{\ast},...,x_{n}^{\ast}\in E^{\ast}$ and $T_{1},...,T_{n}\in \mathcal{F}$, $n\in\mathbb{N}$, we have that 
$$
|\sum_{k=1}^{n}\langle T_{k}x_{k},x_{k}^{\ast}\rangle|\leq \widetilde{C}_{\mathcal{F},\mathcal{S}}\big(\sup_{a_{k}\in \mathbb{D}}\|\sum_{k=1}^{n}a_{k}x_{k}\|_{E}\big)\big(\sup_{b_{k}\in \mathbb{D}}\|\sum_{k=1}^{n}b_{k}x_{k}^{\ast}\|_{E^{\ast}}\big),
$$
where $\widetilde{C}_{\mathcal{F},\mathcal{S}}=C_{\mathcal{F},\mathcal{S}}(\mathrm{Vol}(\Omega))^{\frac{1}{2}}$.
\end{lemma}
\begin{proof}
By Cauchy-Schwarz inequality we have that 
\begin{eqnarray*}
\lefteqn{| \sum_{k=1}^{n} \langle T_{k}x_{k},x_{k}^{\ast}\rangle|=| \int_{\Omega} \langle\sum_{i=1}^{n}e_{i}T_{i}x_{i},\sum_{j=1}^{n}\bar{e}_{j}x_{j}^{\ast}\rangle d\mu|}\\
&\leq&\int_{\Omega}\|\sum_{i=1}^{n}e_{i}T_{i}x_{i}\|_{E}\|\sum_{j=1}^{n}\bar{e}_{j}x_{j}^{\ast}\|_{E^{\ast}}d\mu\\
&\leq&\big(\int_{\Omega}\|\sum_{i=1}^{n}e_{i}T_{i}x_{i}\|_{E}^{2}d\mu\big)^{\frac{1}{2}} \big(\int_{\Omega}\|\sum_{j=1}^{n}\bar{e}_{j}x_{j}^{\ast}\|_{E^{\ast}}^{2}d\mu\big)^{\frac{1}{2}},
\end{eqnarray*}
for certain vectors $e_{1},...,e_{n}$ in $L^{\infty}(\Omega;\mathbb{C},d\mu)$ with $\|e_{k}\|_{L^{\infty}(\Omega;\mathbb{C},d\mu)}\leq 1$, $k\in\{1,...,n\}$, such that $e_{1},...,e_{n}$ are orthonormal in $L^{2}(\Omega;\mathbb{C},d\mu)$. Then, the estimate follows.
\end{proof}

By restricting to the case of resolvents of operators, we can generalize the notion of sectoriality as follows. 

\begin{definition}
Let $E$ be a complex Banach space, $\theta\in[0,\pi)$ and $A\in\mathcal{P}(\theta)$. We say that $A$ is {\em $ON$-sectorial of angle $\theta$}, and denote by $A\in\mathcal{ON}(\theta)$, if the family $\{\lambda(A+\lambda)^{-1}\, |\, \lambda \in S_{\theta}\backslash\{0\}\}$ is $ON$-bounded. In this case, we call the $ON$-bound as {\em $ON$-sectorial bound}.
\end{definition}

Similarly to the sectoriality, the $ON$-sectoriality of an operator is preserved under appropriate shifts and the resulting $ON$-sectorial bound remains uniformly bounded. 

\begin{lemma}\label{easy}
Let $A: \mathcal{D}(A)\rightarrow E$ be $ON$-sectorial of angle $\theta\in(0,\pi)$ with $ON$-sectorial bound $C_{A,\theta}$. If $\omega\in[0,\min\{\theta,\pi-\theta\})$, then for any $c\in S_{\omega}$, $A+c$ is $ON$-sectorial of angle $\theta$ with $ON$-sectorial bound $\leq \frac{C_{A,\theta}}{\sin(\theta+\omega)}$.
\end{lemma}
\begin{proof}
Let $\mathcal{S}=(\Omega,\Sigma,\mu)$ be a finite measure space subject to the $ON$-sectoriality of $A$. For any $\lambda_{1},...,\lambda_{n}\in S_{\theta}\backslash\{0\}$ and $x_{1},...,x_{n}\in E$, $n\in\mathbb{N}$, we have that
\begin{eqnarray*}
\lefteqn{\|\sum_{k=1}^{n}e_{k}\lambda_{k}(A+c+\lambda_{k})^{-1}x_{k}\|_{L^{2}(\Omega;E,d\mu)}}\\
&=&\|\sum_{k=1}^{n}e_{k}(c+\lambda_{k})(A+c+\lambda_{k})^{-1}\frac{\lambda_{k}}{c+\lambda_{k}}x_{k}\|_{L^{2}(\Omega;E,d\mu)}\\
&\leq&C_{A,\theta}\sup_{a_{k}\in \mathbb{D}}\|\sum_{k=1}^{n}a_{k}\frac{\lambda_{k}}{c+\lambda_{k}}x_{k}\|_{E},\\
\end{eqnarray*}
for some vectors $e_{1},...,e_{n}$ in $L^{\infty}(\Omega;\mathbb{C},d\mu)$ with $\|e_{k}\|_{L^{\infty}(\Omega;\mathbb{C},d\mu)}\leq 1$, $k\in\{1,...,n\}$, such that $e_{1},...,e_{n}$ are orthonormal in $L^{2}(\Omega;\mathbb{C},d\mu)$. The result now follows by the estimate
$$
\sup_{\lambda\in S_{\theta}\backslash\{0\}}\Big|\frac{\lambda}{c+\lambda}\Big|\leq \frac{1}{\sin(\theta+\omega)}.
$$
\end{proof}

Next we consider the case of $ON$-boundedness where the vectors involving the estimate are taken from a fixed orthonormal set.

\begin{definition}
Let $E$ be a complex Banach space, $\mathcal{F}\subset\mathcal{L}(E)$ be a family of bounded operators in $E$, $\mathcal{S}=(\Omega,\Sigma,\mu)$ be a finite measure space and $\mathcal{E}=\{e_{k}\}_{k\in\mathbb{N}}$ be a fixed orthonormal set in $L^{2}(\Omega;\mathbb{C},d\mu)$ such that $e_{k}\in L^{\infty}(\Omega;\mathbb{C},d\mu)$ with $\|e_{k}\|_{L^{\infty}(\Omega;\mathbb{C},d\mu)}\leq 1$ for each $k\in\mathbb{N}$. Let $\mathcal{X}_{n}=\{x_{1},...,x_{n}\}$ a finite sequence of vectors in $E$ and $\mathcal{T}_{n}=\{T_{1},...,T_{n}\}$ be a finite sequence of vectors in $\mathcal{F}$, $n\in \mathbb{N}$. We say that $\mathcal{F}$ is {\em $\mathcal{E}$-bounded} if for any triple $\tau=(n,\mathcal{X}_{n},\mathcal{T}_{n})$ there exists a finite sequence $a_{1},...,a_{n}\in \mathbb{D}$ that depends on $\tau$, such that
$$
\|\sum_{k=1}^{n}e_{k}T_{k}x_{k}\|_{L^{2}(\Omega;E,d\mu)}\leq C_{\mathcal{F},\mathcal{E}} \|\sum_{k=1}^{n}a_{k}e_{k}x_{k}\|_{L^{2}(\Omega;E,d\mu)}
$$
for some constant $C_{\mathcal{F},\mathcal{E}}\geq 1$ which is called {\em $\mathcal{E}$-bound} and depends only on $\mathcal{F}$ and $\mathcal{E}$. Furthermore, an operator $A\in\mathcal{P}(\theta)$ in $E$ is called {\em $\mathcal{E}$-sectorial of angle $\theta$}, and we denote by $A\in\mathcal{E}(\theta)$, if the family $\{\lambda(A+\lambda)^{-1}\, |\, \lambda \in S_{\theta}\backslash\{0\}\}$ is $\mathcal{E}$-bounded. In this case, we call the $\mathcal{E}$-bound as {\em $\mathcal{E}$-sectorial bound}.
\end{definition} 

A special example of an $\mathcal{E}$-sectorial operator is any $R$-sectorial operator, i.e. a sectorial operator $A$ such that the family $\{\lambda(A+\lambda)^{-1}\, |\, \lambda \in S_{\theta}\backslash\{0\}\}$ is Rademacher bounded (see Definition \ref{Defrs}). Due to a property of the Rademacher functions, namely the Kahane's contraction principle (see e.g. \cite[Proposition 2.5]{KL1}), in this case the numbers $a_{k}$ in the above definition can be taken equal to one. 

Furthermore, in \cite[Theorem 2.8]{Ro2}, it has been shown that if a sectorial operator $A$ defined in a $UMD$ (unconditionality of martingale differences property) space has bounded imaginary powers with power angle $\phi<\pi$ (see \cite[Section III.4.7]{Am} for definition), then for any $\theta\in(\phi-\pi,\pi-\phi)$ the family $\mathcal{F}_{\theta}=\{(I+re^{-k+i\theta}A)^{-1}\, |\, r\in[\frac{1}{e},1], k\in\mathbb{N}\cup\{0\}\}$ is $\mathcal{E}$-bounded with respect to the orthonormal set $\mathcal{E}=\{e^{ikt}/\sqrt{2\pi}\}_{k\in\mathbb{N}\cup\{0\}}$ in $L^{2}(0,2\pi)$. This together with \cite[Theorem 2.6]{Ro2} showed for example that the sum of two resolvent commuting operators in a $UMD$ space such that one admits a bounded $H^{\infty}$-calculus and the other one has bounded imaginary powers is closed and invertible, provided that the standard parabolic condition between the corresponding angles is satisfied. 

The class of $\mathcal{E}$-sectorial operators behaves nicely in relatively small perturbations as we can see by the following.

\begin{proposition}\label{perp}
Let $E$ be a complex Banach space and $A\in\mathcal{E}(\theta)$ with sectorial bound equal to $M$ and $\mathcal{E}$-sectorial bound equal to $C_{A,\theta}$. Let also $B$ be a linear operator in $E$ such that $\mathcal{D}(A)\subseteq \mathcal{D}(B)$ and 
$$
\|BA^{-1}\|_{\mathcal{L}(E)}<\min\{\frac{1}{1+C_{A,\theta}},\frac{1}{1+M}\}. 
$$
Then $A+B\in\mathcal{E}(\theta)$ and the $\mathcal{E}$-sectorial bound of $A+B$ is $\leq C_{A,\theta}/(1-(1+C_{A,\theta})\|BA^{-1}\|_{\mathcal{L}(E)})$.
\end{proposition}
\begin{proof}
Clearly, the family $\{A(A+\lambda)^{-1}\, |\, \lambda \in S_{\theta}\backslash\{0\}\}$ is also $\mathcal{E}$-bounded with $\mathcal{E}$-bound $\leq 1+C_{A,\theta}$. Moreover, $S_{\theta}\subset\rho(-(A+B))$ and the resolvent of $A+B$ is given by the following absolutely converged Neumann series 
$$
(A+B+\lambda)^{-1}=(A+\lambda)^{-1}\sum_{k=0}^{\infty}(-1)^{k}(B(A+\lambda)^{-1})^{k}, \quad \lambda\in S_{\theta}.
$$
By using successively the $\mathcal{E}$-boundedness of $A(A+\lambda)^{-1}$ together with the compactness of $\mathbb{D}$ we obtain the result. 
\end{proof}

Let $E$ be a complex Banach space and $A\in\mathcal{L}(\mathcal{D}(A),E)$ with $0\notin\sigma(A)$. In this situation, instead of using the graph norm, for simplicity we equip $\mathcal{D}(A)$ with $\|A\cdot\|_{E}$. By denoting with $[\cdot,\cdot]$ the commutation operation, we introduce next our basic commutation condition between two sectorial operators.

\begin{condition}\label{c3}
Let $A:\mathcal{D}(A)\rightarrow E$, $B:\mathcal{D}(B)\rightarrow E$ be linear operators in a complex Banach space $E$ such that $A\in\mathcal{P}(\theta_{A})$, $B\in\mathcal{P}(\theta_{B})$ and 
$$
S_{\theta_{B}}\ni \mu\mapsto (B+\mu)^{-1}\in \mathcal{L}(\mathcal{D}(A))
$$
is a well-defined Lebesgue measurable map. Assume that there exist some constants $C>0$ and $\alpha_{j},\beta_{j}\geq0$, $j\in\{1,2,3\}$, such that
$$
\|[(A+\lambda)^{-1},(B+\mu)^{-1}]\|_{\mathcal{L}(X_{j},Y_{j})}\leq\frac{C}{(1+|\lambda|^{\alpha_{j}})(1+|\mu|^{\beta_{j}})} \quad \text{when} \quad (\lambda,\mu)\in S_{\theta_{A}}\times S_{\theta_{B}},
$$
where
\begin{itemize}
\item[{(i)}] $X_{1}=Y_{1}=E$,\quad $\alpha_{1}+\beta_{1}>2$, \quad $\alpha_{1}>0$, $\beta_{1}>0$.
\item[{(ii)}] $X_{2}=E$, $Y_{2}=\mathcal{D}(A)$,\quad $\alpha_{2}+\beta_{2}>1$, \quad $\alpha_{2}\geq0$, $\beta_{2}>0$.
\item[{(iii)}] $X_{3}=Y_{3}=\mathcal{D}(A)$,\quad $\alpha_{3}+\beta_{3}>2$, \quad $\alpha_{3}\geq0$, $\beta_{3}>0$.
\end{itemize}
\end{condition}

Note that the above condition implies density of $\mathcal{D}(A)\cap\mathcal{D}(B)$ in $E$, since by \cite[Proposition 9.4]{KL1} for any $x\in E$ we have that $(\delta B+1)^{-1}(\delta A+1)^{-1}x\rightarrow x$ strongly as $\delta\rightarrow+0$. Moreover, in view e.g. of the following formally written equalities
\begin{eqnarray}\nonumber
\lefteqn{[(A+\lambda)^{-1},(B+\mu)^{-1}]}\\\label{c31}
&=&(A+\lambda)^{-1}(B+\mu)^{-1}[A,B]B^{-(1+\nu)}B(B+\mu)^{-1}B^{\nu}A^{-\eta}A^{\eta}(A+\lambda)^{-1},
\end{eqnarray}
\begin{eqnarray}\nonumber
\lefteqn{A[(A+\lambda)^{-1},(B+\mu)^{-1}]A^{-1}}\\\label{c32}
&=&A(A+\lambda)^{-1}(B+\mu)^{-1}[A,B]B^{-(1+\gamma)}B^{1-\xi}(B+\mu)^{-1}B^{\gamma+\xi}A^{-1}(A+\lambda)^{-1},
\end{eqnarray}
for certain $\nu,\eta,\gamma,\xi\in(0,1)$, by Lemma \ref{l1} we see that Condition \ref{c3} is fulfilled when the commutator $[A,B]$ is of {\em lower order} e.g. this can happen in the case of differential operators. 

Let us now consider one further condition that turns out to be stronger than Condition \ref{c3}. 

\begin{condition}\label{c4}
Let $A:\mathcal{D}(A)\rightarrow E$, $B:\mathcal{D}(B)\rightarrow E$ be linear operators in a complex Banach space $E$ such that $A\in\mathcal{P}(\theta_{A})$, $B\in\mathcal{P}(\theta_{B})$ and
$$
S_{\theta_{B}}\ni \mu\mapsto (B+\mu)^{-1}\in \mathcal{L}(\mathcal{D}(A))
$$
is a well-defined Lebesgue measurable map. Assume that there exist some constants $C>0$ and $\alpha_{1}, \beta_{1}, \beta_{2}>0$, $\alpha_{2}\geq0$ with $\alpha_{1}+\beta_{1}>2$ and $ \alpha_{2}+\beta_{2}>1$, such that for all $(\lambda,\mu)\in S_{\theta_{A}}\times S_{\theta_{B}}$ we have
$$
\|[(A+\lambda)^{-1},(B+\mu)^{-1}]\|_{\mathcal{L}(E)}\leq \frac{C}{(1+|\lambda|^{\alpha_{1}})(1+|\mu|^{\beta_{1}})}
$$
and
$$
\|(A+\lambda)^{-1}[(B+\mu)^{-1},A]\|_{\mathcal{L}(\mathcal{D}(A))}\leq \frac{C}{(1+|\lambda|^{\alpha_{2}})(1+|\mu|^{\beta_{2}})}.
$$
\end{condition} 

In view of the identity
\begin{gather}\label{com1}
[(A+\lambda)^{-1},(B+\mu)^{-1}]=(A+\lambda)^{-1}[(B+\mu)^{-1},A](A+\lambda)^{-1},
\end{gather}
where $(\lambda,\mu)\in S_{\theta_{A}}\times S_{\theta_{B}}$, we obtain the following.

\begin{remark}\label{rcomp}
Condition \ref{c4} implies Condition \ref{c3}. However, Condition \ref{c4} does not imply the Da Prato-Grisvard condition \cite[(3.1)]{PS} or the Labbas-Terreni condition \cite[(3.2)]{PS}. Moreover, condition \cite[(3.1)]{PS} or condition \cite[(3.2)]{PS} implies the case (i) of Condition \ref{c3}, which is an underlying commutation condition that appears in \cite{MP} and \cite{PS}, as well as in our calculations. By \eqref{com1}, condition \cite[(3.1)]{PS} also implies the case (ii) of Condition \ref{c3}. Finally, from \eqref{com1} we see that the cases (ii) and (iii) of Condition \ref{c3} are of similar flexibility.
\end{remark}

We recall next the Da Prato and Grisvard formula for the inverse of the closure of the sum of two resolvent commuting sectorial operators.

\begin{notation}
Let $E$ be a complex Banach space and $A$, $B$ be linear operators in $E$ such that $A\in\mathcal{P}(\theta_{A})$ and $B\in\mathcal{P}(\theta_{B})$ with $\theta_{A}+\theta_{B}>\pi$. Let $\psi\in[0,\pi-\max\{\theta_{A},\theta_{B}\})$, $c\in S_{\psi}$, and consider the bounded operators $K_{c}, L_{c}\in\mathcal{L}(E)$ defined by
\begin{gather*}
K_{c}=\frac{1}{2\pi i}\int_{\Gamma_{\theta_{B}}}(A-z)^{-1}(B_{c}+z)^{-1}dz \quad \text{and} \quad L_{c}=\frac{1}{2\pi i}\int_{\Gamma_{\theta_{B}}}(B_{c}+z)^{-1}(A-z)^{-1}dz,
\end{gather*}
where $B_{c}=B+c$. By a sectoriality area extension argument we can replace the path $\Gamma_{\theta_{B}}$ in the above formulas by $\Gamma_{\rho,\theta_{B}}$ or by $\pm\delta+\Gamma_{\theta_{B}-\varepsilon}$, for sufficiently small $\rho,\delta,\varepsilon>0$.
\end{notation}

\begin{remark}\label{r6}
In the definition of $K_{c}$, if we keep $\theta_{B}$ fixed and replace $B$ by a family of operators $B(\xi)\in \mathcal{P}(\theta_{B})$, $\xi\in \Xi$, indexed by a set $\Xi$, such that the sectorial bounds of $B(\xi)$ are uniformly bounded in $\xi\in \Xi$, then we can still replace $\Gamma_{\theta_{B}}$ by $\Gamma_{\rho,\theta_{B}}$ for some fixed $\rho>0$ independent of $\xi$. 
\end{remark}

We recall the following mapping property of the operator $K_{c}$. 

\begin{lemma}\label{l32}
Let $E$ be a complex Banach space and $A$, $B$ be linear operators in $E$ such that $A\in\mathcal{P}(\theta_{A})$ and $B\in\mathcal{P}(\theta_{B})$ with $\theta_{A}+\theta_{B}>\pi$. Then, the operator $K_{c}$ maps $\mathcal{D}(B)$ to $\mathcal{D}(A)$.
\end{lemma}
\begin{proof}
If $w\in\mathbb{C}$ with $\mathrm{Re}(w)<0$ then
\begin{eqnarray*}
\lefteqn{K_{c}B^{w}=\frac{1}{2\pi i}\int_{\Gamma_{\theta_{B}-\varepsilon}}(A-z)^{-1}\big(\frac{1}{2\pi i}\int_{-\delta+\Gamma_{\theta_{B}}}(-\lambda)^{w}(B_{c}+z)^{-1}(B+\lambda)^{-1}d\lambda\big)dz}\\
&=&\frac{1}{2\pi i}\int_{\Gamma_{\theta_{B}-\varepsilon}}(A-z)^{-1}\big(\frac{1}{2\pi i}\int_{-\delta+\Gamma_{\theta_{B}}}(-\lambda)^{w}(\lambda-c-z)^{-1}((B_{c}+z)^{-1}-(B+\lambda)^{-1})d\lambda\big)dz\\
&=&\frac{1}{(2\pi i)^{2}}\int_{\Gamma_{\theta_{B}-\varepsilon}}\int_{-\delta+\Gamma_{\theta_{B}}}(-\lambda)^{w}(\lambda-c-z)^{-1}(A-z)^{-1}(B_{c}+z)^{-1}d\lambda dz\\
&&-\frac{1}{(2\pi i)^{2}}\int_{-\delta+\Gamma_{\theta_{B}}}\int_{\Gamma_{\theta_{B}-\varepsilon}}(-\lambda)^{w}(\lambda-c-z)^{-1}(A-z)^{-1}(B+\lambda)^{-1}dzd\lambda,
\end{eqnarray*}
where we have used Fubini's theorem. By Cauchy's theorem the first term on the right hand side of the above equation is zero. Therefore, 
\begin{gather}\label{ef}
K_{c}B^{w}=\frac{1}{2\pi i}\int_{-\delta+\Gamma_{\theta_{B}}}(-\lambda)^{w}(A+c-\lambda)^{-1}(B+\lambda)^{-1}d\lambda.
\end{gather}
Since the integral 
\begin{gather*}
\int_{-\delta+\Gamma_{\theta_{B}}}(-\lambda)^{w}A(A+c-\lambda)^{-1}(B+\lambda)^{-1}d\lambda
\end{gather*}
converges absolutely, by \eqref{ef} we have that $K_{c}B^{w}$ maps $E$ to $\mathcal{D}(A)$.
\end{proof}

Finally, we recall the following commutation formula.

\begin{lemma}\label{com}
Let $E$ be a complex Banach space and $A$, $B$ be linear operators in $E$ with $A\in\mathcal{P}(\theta_{A})$. Then, for any $f\in H_{0}^{\infty}(\phi)$, $\phi\in[0,\theta_{A})$, and any $\lambda\in \rho(-B)\neq \emptyset$ we have that
$$
[f(-A),(B+\lambda)^{-1}]=Q_{f}(\lambda),
$$
where
$$
Q_{f}(\lambda)=\frac{1}{2\pi i}\int_{\Gamma_{\theta_{A}}}f(z)[(A+z)^{-1},(B+\lambda)^{-1}]dz \in\mathcal{L}(E).
$$
\end{lemma}
\begin{proof}
Follows directly by the integral formula for the functional calculus of $A$.
\end{proof}

\section{The sum of non-commuting operators}

In this section we consider sums of possibly non-commuting operators satisfying Condition \ref{c3} and show closedness and invertibility. Firstly, we perturb the Da Prato and Grisvard formula from the left in order to construct an unbounded left inverse of the sum. Then, we do similar perturbation from the right and construct an approximation of the right inverse. Finally, by employing the extra properties of the boundedness of the $H^\infty$-calculus and the $ON$-sectoriality we show that the above two constructions give the inverse of the sum. We start by applying certain fractional powers to the Da Prato and Grisvard formula as follows.

\begin{proposition}\label{t1}
Let $E$ be a complex Banach space and $A$, $B$ be linear operators in $E$ such that $A\in\mathcal{P}(\theta_{A})$ and $B\in\mathcal{P}(\theta_{B})$ with $\theta_{A}+\theta_{B}>\pi$. If Condition \ref{c3} is satisfied, then the operator $K_{c}$ maps the range $\mathrm{Ran}(A+B_{c})$ to $\mathcal{D}(A)$ and there exists some $P_{c} \in \mathcal{L}(E)$ such that 
$$
AK_{c}(A+B_{c})=(I+P_{c})A \quad \text{in} \quad \mathcal{D}(A+B).
$$
Furthermore, $\|P_{c}\|_{\mathcal{L}(E)}\rightarrow0$ when $|c|\rightarrow\infty$.
\end{proposition}

\begin{proof}
Let $w\in\mathbb{C}$ with $\mathrm{Re}(w)<0$. We have
\begin{eqnarray*}
\lefteqn{A^{w}K_{c}=\frac{1}{2\pi i}\int_{\Gamma_{\theta_{B}}}A^{w}(A-z)^{-1}(B_{c}+z)^{-1}dz}\\
&=&\frac{1}{2\pi i}\int_{\Gamma_{\theta_{B}}}\big(\frac{1}{2\pi i}\int_{-\delta+\Gamma_{\theta_{A}}}(-\lambda)^{w}(A+\lambda)^{-1}(A-z)^{-1}d\lambda\big)(B_{c}+z)^{-1}dz\\
&=&\frac{1}{(2\pi i)^{2}}\int_{\Gamma_{\theta_{B}}}\big(\int_{-\delta+\Gamma_{\theta_{A}}}(-\lambda)^{w}(\lambda+z)^{-1}((A-z)^{-1}-(A+\lambda)^{-1})d\lambda\big)(B_{c}+z)^{-1}dz\\
&=&\frac{1}{(2\pi i)^{2}}\int_{\Gamma_{\theta_{B}}}\int_{-\delta+\Gamma_{\theta_{A}}}(-\lambda)^{w}(\lambda+z)^{-1}(A-z)^{-1}(B_{c}+z)^{-1}d\lambda dz\\
&&-\frac{1}{(2\pi i)^{2}}\int_{-\delta+\Gamma_{\theta_{A}}}\int_{\Gamma_{\theta_{B}}}(-\lambda)^{w}(\lambda+z)^{-1}(A+\lambda)^{-1}(B_{c}+z)^{-1}dzd\lambda,
\end{eqnarray*}
where at the last step we have used Fubini's theorem. By Cauchy's theorem, the first term on the right hand side of the above equation is zero. Therefore
\begin{gather*}
A^{w}K_{c}=\frac{1}{2\pi i}\int_{-\delta+\Gamma_{\theta_{A}}}(-\lambda)^{w}(A+\lambda)^{-1}(B_{c}-\lambda)^{-1}d\lambda.
\end{gather*}
Hence, if $\theta\in(0,1)$ by Cauchy's theorem we obtain
\begin{gather}\label{aak}
A^{-\theta}K_{c}=\frac{1}{2\pi i}\int_{\Gamma_{\theta_{B}}}\lambda^{-\theta}(A-\lambda)^{-1}(B_{c}+\lambda)^{-1}d\lambda.
\end{gather}

If $x\in\mathcal{D}(A+B)$, then we estimate 
\begin{eqnarray*}
\lefteqn{A^{-\theta}K_{c}(A+B_{c})x=\frac{1}{2\pi i}\int_{\Gamma_{\theta_{B}}}\lambda^{-\theta}(A-\lambda)^{-1}B_{c}(B_{c}+\lambda)^{-1}xd\lambda}\\
&&+\frac{1}{2\pi i}\int_{\Gamma_{\theta_{B}}}\lambda^{-\theta}(B_{c}+\lambda)^{-1}A(A-\lambda)^{-1}xd\lambda+\frac{1}{2\pi i}\int_{\Gamma_{\theta_{B}}}\lambda^{-\theta}[(A-\lambda)^{-1},(B_{c}+\lambda)^{-1}]Axd\lambda\\
&=&\frac{1}{2\pi i}\int_{\Gamma_{\theta_{B}}}\lambda^{-\theta}(A-\lambda)^{-1}xd\lambda-\frac{1}{2\pi i}\int_{\Gamma_{\theta_{B}}}\lambda^{1-\theta}(A-\lambda)^{-1}(B_{c}+\lambda)^{-1}xd\lambda\\
&&+\frac{1}{2\pi i}\int_{\Gamma_{\theta_{B}}}\lambda^{-\theta}(B_{c}+\lambda)^{-1}xd\lambda+\frac{1}{2\pi i}\int_{\Gamma_{\theta_{B}}}\lambda^{1-\theta}(B_{c}+\lambda)^{-1}(A-\lambda)^{-1}xd\lambda\\
&&+\frac{1}{2\pi i}\int_{\Gamma_{\theta_{B}}}\lambda^{-\theta}[(A-\lambda)^{-1},(B_{c}+\lambda)^{-1}]Axd\lambda\\
&=&A^{-\theta}x-\frac{1}{2\pi i}\int_{\Gamma_{\theta_{B}}}\lambda^{1-\theta}[(A-\lambda)^{-1},(B_{c}+\lambda)^{-1}]xd\lambda+\frac{1}{2\pi i}\int_{\Gamma_{\theta_{B}}}\lambda^{-\theta}[(A-\lambda)^{-1},(B_{c}+\lambda)^{-1}]Axd\lambda.
\end{eqnarray*}

By taking the limit in the above equation as $\theta\rightarrow 0$, since $A^{-\theta}\rightarrow I$ strongly, by the dominated convergence theorem we obtain
\begin{eqnarray}\nonumber
\lefteqn{K_{c}(A+B_{c})x=x}\\\label{redo}
&&-\frac{1}{2\pi i}\int_{\Gamma_{\theta_{B}}}\lambda[(A-\lambda)^{-1},(B_{c}+\lambda)^{-1}]xd\lambda+\frac{1}{2\pi i}\int_{\Gamma_{\theta_{B}}}[(A-\lambda)^{-1},(B_{c}+\lambda)^{-1}]Axd\lambda,
\end{eqnarray}
where we have employed Condition \ref{c3} for the existence of the dominant and for the absolute convergence of the above integrals. 

Since by Condition \ref{c3} the integrals 
$$
\int_{\Gamma_{\theta_{B}}}\lambda A[(A-\lambda)^{-1},(B_{c}+\lambda)^{-1}]xd\lambda \quad \text{and} \quad \int_{\Gamma_{\theta_{B}}}A[(A-\lambda)^{-1},(B_{c}+\lambda)^{-1}]Axd\lambda
$$
converge absolutely, \eqref{redo} implies that $K_{c}(A+B_{c})$ maps $\mathcal{D}(A+B)$ to $\mathcal{D}(A)$ and 
\begin{eqnarray}\nonumber
\lefteqn{AK_{c}(A+B_{c})x}\\\label{ktm}
&=&\big(I-\frac{1}{2\pi i}\int_{\Gamma_{\theta_{B}}}\lambda A[(A-\lambda)^{-1},(B_{c}+\lambda)^{-1}]A^{-1}d\lambda+\frac{1}{2\pi i}\int_{\Gamma_{\theta_{B}}}A[(A-\lambda)^{-1},(B_{c}+\lambda)^{-1}]d\lambda\big)Ax
\end{eqnarray}
for all $x\in\mathcal{D}(A+B)$. 

Furthermore, by Condition \ref{c3}, the norm of 
\begin{eqnarray}\nonumber
\lefteqn{P_{c}=-\frac{1}{2\pi i}\int_{\Gamma_{\theta_{B}}}\lambda A[(A-\lambda)^{-1},(B_{c}+\lambda)^{-1}]A^{-1}d\lambda}\\\label{Pc}
&&+\frac{1}{2\pi i}\int_{\Gamma_{\theta_{B}}}A[(A-\lambda)^{-1},(B_{c}+\lambda)^{-1}]d\lambda \in \mathcal{L}(E)
\end{eqnarray}
becomes arbitrary small by taking $|c|$ sufficiently large. 
 \end{proof}

Similarly, we can build an approximation of the right inverse of the sum by applying the Da Prato and Grisvard formula to certain fractional powers as follows.

\begin{proposition}\label{t12}
Let $E$ be a complex Banach space and $A$, $B$ be linear operators in $E$ such that $A\in\mathcal{P}(\theta_{A})$ and $B\in\mathcal{P}(\theta_{B})$ with $\theta_{A}+\theta_{B}>\pi$. If Condition \ref{c3} is satisfied, then the operator $L_{c}$ maps $\mathcal{D}(A)$ to $\mathcal{D}(A+B)$ and there exists some $T_{c} \in \mathcal{L}(E)$ such that
$$
(A+B_{c})L_{c}=I+T_{c} \quad\, \text{in} \quad \mathcal{D}(A).
$$
Furthermore, $\|T_{c}\|_{\mathcal{L}(E)}\rightarrow 0$ when $|c|\rightarrow\infty$.
\end{proposition}
\begin{proof}
We need an analogue of formula \eqref{ktm}. If $w\in\mathbb{C}$ with $\mathrm{Re}(w)<0$, then
\begin{eqnarray*}
\lefteqn{L_{c}A^{w}=\frac{1}{2\pi i}\int_{-\delta+\Gamma_{\theta_{B}}}(B_{c}+z)^{-1}\big(\frac{1}{2\pi i}\int_{-\delta+\Gamma_{\theta_{A}}}(-\lambda)^{w}(A-z)^{-1}(A+\lambda)^{-1}d\lambda\big)dz}\\
&=&\frac{1}{2\pi i}\int_{-\delta+\Gamma_{\theta_{B}}}(B_{c}+z)^{-1}\big(\frac{1}{2\pi i}\int_{-\delta+\Gamma_{\theta_{A}}}(-\lambda)^{w}(\lambda+z)^{-1}((A-z)^{-1}-(A+\lambda)^{-1})d\lambda\big)dz\\
&=&(\frac{1}{2\pi i})^{2}\int_{-\delta+\Gamma_{\theta_{B}}}\int_{-\delta+\Gamma_{\theta_{A}}}(-\lambda)^{w}(\lambda+z)^{-1}(B_{c}+z)^{-1}(A-z)^{-1}d\lambda dz\\
&&-(\frac{1}{2\pi i})^{2}\int_{-\delta+\Gamma_{\theta_{A}}}\int_{-\delta+\Gamma_{\theta_{B}}}(-\lambda)^{w}(\lambda+z)^{-1}(B_{c}+z)^{-1}(A+\lambda)^{-1}dzd\lambda,
\end{eqnarray*}
where we have used Fubini's theorem. By Cauchy's theorem the first term on the right hand side of the above equation is zero. Therefore,
\begin{eqnarray}\label{kklt}
L_{c}A^{w}=\frac{1}{2\pi i}\int_{-\delta+\Gamma_{\theta_{A}}}(-\lambda)^{w}(B_{c}-\lambda)^{-1}(A+\lambda)^{-1}d\lambda.
\end{eqnarray}

Let $\theta\in(0,1)$. Since the integral 
$$
\int_{\Gamma_{\theta_{A}}}(-\lambda)^{-\theta}B_{c}(B_{c}-\lambda)^{-1}(A+\lambda)^{-1}d\lambda
$$
converges absolutely, by \eqref{kklt} we have that $L_{c}A^{-\theta}$ maps $E$ to $\mathcal{D}(B)$ and 
\begin{eqnarray}\nonumber
\lefteqn{B_{c}L_{c}A^{-\theta}=\frac{1}{2\pi i}\int_{\Gamma_{\theta_{A}}}(-\lambda)^{-\theta}(B_{c}-\lambda+\lambda)(B_{c}-\lambda)^{-1}(A+\lambda)^{-1}d\lambda}\\\nonumber
&=&\frac{1}{2\pi i}\int_{\Gamma_{\theta_{A}}}(-\lambda)^{-\theta}(A+\lambda)^{-1}d\lambda-\frac{1}{2\pi i}\int_{\Gamma_{\theta_{A}}}(-\lambda)^{1-\theta}(B_{c}-\lambda)^{-1}(A+\lambda)^{-1}d\lambda\\\label{kklt2}
&=&A^{-\theta}-\frac{1}{2\pi i}\int_{\Gamma_{\theta_{A}}}(-\lambda)^{1-\theta}(B_{c}-\lambda)^{-1}(A+\lambda)^{-1}d\lambda.
\end{eqnarray}

Moreover, by Condition \ref{c3}, the integrals 
$$
\int_{\Gamma_{\theta_{A}}}(-\lambda)^{-\theta}A(A+\lambda)^{-1}(B_{c}-\lambda)^{-1}d\lambda \quad \text{and} \quad \int_{\Gamma_{\theta_{A}}}(-\lambda)^{-\theta}A[(A+\lambda)^{-1},(B_{c}-\lambda)^{-1}]d\lambda
$$
converge absolutely. Hence, by \eqref{kklt} we find that $L_{c}A^{-\theta}$ maps $E$ to $\mathcal{D}(A)$ and 
\begin{eqnarray}\nonumber
\lefteqn{AL_{c}A^{-\theta}=\frac{1}{2\pi i}\int_{\Gamma_{\theta_{A}}}(-\lambda)^{-\theta}(A+\lambda-\lambda)(A+\lambda)^{-1}(B_{c}-\lambda)^{-1}d\lambda}\\\nonumber
&& -\frac{1}{2\pi i}\int_{\Gamma_{\theta_{A}}}(-\lambda)^{-\theta}A[(A+\lambda)^{-1},(B_{c}-\lambda)^{-1}]d\lambda\\\label{ast}
&=&\frac{1}{2\pi i}\int_{\Gamma_{\theta_{A}}}(-\lambda)^{1-\theta}(A+\lambda)^{-1}(B_{c}-\lambda)^{-1}d\lambda-\frac{1}{2\pi i}\int_{\Gamma_{\theta_{A}}}(-\lambda)^{-\theta}A[(A+\lambda)^{-1},(B_{c}-\lambda)^{-1}]d\lambda.
\end{eqnarray}

Therefore, by \eqref{kklt2} and \eqref{ast} we obtain 
\begin{eqnarray*}
\lefteqn{(A+B_{c})L_{c}A^{-\theta}=A^{-\theta}}\\
&&+\frac{1}{2\pi i}\int_{\Gamma_{\theta_{A}}}(-\lambda)^{1-\theta}[(A+\lambda)^{-1},(B_{c}-\lambda)^{-1}]d\lambda-\frac{1}{2\pi i}\int_{\Gamma_{\theta_{A}}}(-\lambda)^{-\theta}A[(A+\lambda)^{-1},(B_{c}-\lambda)^{-1}]d\lambda.
\end{eqnarray*}
Hence, if $x\in\mathcal{D}(A)$ we find that
\begin{eqnarray*}
\lefteqn{(A+B_{c})L_{c}x=x}\\
&&+\frac{1}{2\pi i}\int_{\Gamma_{\theta_{A}}}(-\lambda)^{1-\theta}[(A+\lambda)^{-1},(B_{c}-\lambda)^{-1}]A^{\theta}xd\lambda-\frac{1}{2\pi i}\int_{\Gamma_{\theta_{A}}}(-\lambda)^{-\theta}A[(A+\lambda)^{-1},(B_{c}-\lambda)^{-1}]A^{\theta}xd\lambda.
\end{eqnarray*}

By the Dunford integral formula for the complex powers and the dominated convergence theorem we have that $A^{\theta}x=A^{\theta-1}Ax\rightarrow x$ when $\theta\rightarrow0$, for any $x\in\mathcal{D}(A)$. Thus, by taking the pointwise limit in the above equation, by the dominated convergence theorem we obtain that 
\begin{eqnarray*}
\lefteqn{(A+B_{c})L_{c}x}\\
&=&\big(I-\frac{1}{2\pi i}\int_{\Gamma_{\theta_{A}}}\lambda[(A+\lambda)^{-1},(B_{c}-\lambda)^{-1}]d\lambda-\frac{1}{2\pi i}\int_{\Gamma_{\theta_{A}}}A[(A+\lambda)^{-1},(B_{c}-\lambda)^{-1}]d\lambda\big)x
\end{eqnarray*}
for all $x\in\mathcal{D}(A)$, where we have used Condition \ref{c3} for the existence of the dominant and for the absolute convergence of the last integrals in the operator norm. 

Finally, by Condition \ref{c3}, the norm of
\begin{eqnarray}\nonumber
\lefteqn{T_{c}=-\frac{1}{2\pi i}\int_{\Gamma_{\theta_{A}}}\lambda[(A+\lambda)^{-1},(B_{c}-\lambda)^{-1}]d\lambda}\\\label{Tc}
&&-\frac{1}{2\pi i}\int_{\Gamma_{\theta_{A}}}A[(A+\lambda)^{-1},(B_{c}-\lambda)^{-1}]d\lambda \in \mathcal{L}(E)
\end{eqnarray}
becomes arbitrary small by taking $|c|$ large enough.
\end{proof}

We are now in the position to impose further assumptions to our operators and make the above constructed perturbations to serve as a left and a right inverse of the sum. The main task is to show that the left inverse approximation given by Proposition \ref{t1} becomes a bounded operator. We manage this by decomposing dyadically the integral representation of the unbounded part and then using the consequences of the $ON$-boundedness and the boundedness of the $H^\infty$-calculus.

\begin{theorem}\label{t2}
Let $E$ be a complex Banach space and $A$, $B$ be linear operators in $E$ such that $A\in\mathcal{H}^{\infty}(\theta_{A})$ and $B\in\mathcal{ON}(\theta_{B})$ with $\theta_{A}+\theta_{B}>\pi$. If Condition \ref{c3} is satisfied, then $A+B$ is closed and there exists a constant $c_{0}\geq0$ such that $\sigma(A+B+c_{0})\in S_{\pi-\min\{\theta_{A},\theta_{B}\}}$. Furthermore, for any $\omega\in[0,\min\{\theta_{A},\theta_{B}\})$ we have that $A+B+c_{0}\in\mathcal{P}(\omega)$.
\end{theorem}
\begin{proof}
Let $\psi\in[0,\pi-\max\{\theta_{A},\theta_{B}\})$, $c\in S_{\psi}$ and $\psi_{A}=\theta_{A}-\varepsilon$ with $\varepsilon>0$ sufficiently small. Since the integral 
\begin{gather*}
\int_{-\Gamma_{\psi_{A}}}\lambda^{-\theta}A(A-\lambda)^{-1}(B_{c}+\lambda)^{-1}d\lambda
\end{gather*}
converges absolutely when $\theta\in(0,1)$, by \eqref{aak} we have that $A^{-\theta}K_{c}$ maps $E$ to $\mathcal{D}(A)$ and 
\begin{eqnarray*}
\lefteqn{AA^{-\theta}K_{c}=\frac{1}{2\pi i}\int_{-\Gamma_{\psi_{A}}}\lambda^{-\theta}A(A-\lambda)^{-1}(B_{c}+\lambda)^{-1}d\lambda}\\
&=&\frac{1}{2\pi i}\int_{-\Gamma_{\psi_{A}}}\lambda^{-\theta}(B_{c}+\lambda)^{-1}d\lambda+\frac{1}{2\pi i}\int_{-\Gamma_{\psi_{A}}}\lambda^{1-\theta}(A-\lambda)^{-1}(B_{c}+\lambda)^{-1}d\lambda\\
&=&\frac{1}{2\pi i}\int_{-\Gamma_{\psi_{A}}}\lambda^{1-\theta}(A-\lambda)^{-1}(B_{c}+\lambda)^{-1}d\lambda.
\end{eqnarray*}
Therefore, by replacing $\theta$ with $\theta+\phi$, with the further restriction $\theta,\phi\in(0,1)$ such that $\theta+\phi<1$, and then applying $A^{\phi}$ to the above equation, we find that 
\begin{gather}\label{ppl}
AA^{-\theta}K_{c}=U_{\theta}+G_{\theta}
\end{gather}
with
\begin{gather*}
U_{\theta}=\frac{1}{2\pi i}\int_{-\Gamma_{\psi_{A}}\cap \mathbb{D}}\lambda^{1-(\theta+\phi)}A^{\phi}(A-\lambda)^{-1}(B_{c}+\lambda)^{-1}d\lambda
\end{gather*}
and
\begin{gather*}
G_{\theta}=\frac{1}{2\pi i}\int_{-\Gamma_{\psi_{A}}\cap(\mathbb{C}\backslash \mathbb{D})}\lambda^{1-(\theta+\phi)}A^{\phi}(A-\lambda)^{-1}(B_{c}+\lambda)^{-1}d\lambda,
\end{gather*}
where we have used Lemma \ref{l1} for the absolute convergence of the above integral. 

For any $m\in\mathbb{N}$ define
\begin{eqnarray}\nonumber
\lefteqn{G_{m}^{\pm} = \frac{e^{\pm i(\pi-\psi_{A})(2-(\theta+\phi))}}{2\pi i}\int_{1}^{2^{m}}A^{\phi}(A-re^{\pm i(\pi-\psi_{A})})^{-1}(B_{c}+re^{\pm i(\pi-\psi_{A})})^{-1}r^{1-(\theta+\phi)}dr}\\\nonumber
&=&\frac{e^{\pm i(\pi-\psi_{A})(2-(\theta+\phi))}}{2\pi i} \sum_{k=0}^{m-1}\int_{2^{k}}^{2^{k+1}}A^{\phi}(A+re^{\mp i\psi_{A}})^{-1}(B_{c}-re^{\mp i\psi_{A}})^{-1}r^{1-(\theta+\phi)}dr\\\nonumber
&=&\frac{e^{\pm i(\pi-\psi_{A})(2-(\theta+\phi))}}{2\pi i}\sum_{k=0}^{m-1}\int_{1}^{2}A^{\phi}(A+t2^{k}e^{\mp i\psi_{A}})^{-1}(B_{c}-t2^{k}e^{\mp i\psi_{A}})^{-1}t^{1-(\theta+\phi)}2^{k(2-(\theta+\phi))}dt\\\label{gm}
& =& \frac{e^{\pm i(\psi_{A}-\pi)(\theta+\phi)}}{2\pi i}\int_{1}^{2}W_{m}^{\pm}(t)\frac{dt}{t},
\end{eqnarray}
where 
$$
W_{m}^{\pm}(t)=\sum_{k=0}^{m-1}A^{\phi}(A+t2^{k}e^{\mp i\psi_{A}})^{-1}(B_{c}-t2^{k}e^{\mp i\psi_{A}})^{-1}t^{2-(\theta+\phi)}2^{k(2-(\theta+\phi))}e^{\mp i2\psi_{A}},
$$
and take any $x\in E$, $x^{\ast}\in E^{\ast}$. 

We proceed now as in \cite{KW1} in order to obtain uniform estimates for the above operator families. More precisely we have
\begin{eqnarray*}
\lefteqn{|\langle W_{m}^{\pm}(t)x,x^{\ast}\rangle|}\\
&=&|\sum_{k=0}^{m-1}\langle t^{\frac{1-(\theta+\phi)}{p}}2^{\frac{k}{p}(1-(\theta+\phi))}e^{\mp i\frac{\psi_{A}}{p}}\big(A^{\phi}(A+t2^{k}e^{\mp i\psi_{A}})^{-1}\big)^{\frac{1}{p}}(B_{c}-t2^{k}e^{\mp i\psi_{A}})^{-1}\\
&&\times t2^{k}e^{\mp i\psi_{A}}x,t^{\frac{1-(\theta+\phi)}{q}}2^{\frac{k}{q}(1-(\theta+\phi))}e^{\mp i\frac{\psi_{A}}{q}}\big((A^{\ast})^{\phi}(A^{\ast}+t2^{k}e^{\mp i\psi_{A}})^{-1}\big)^{\frac{1}{q}}x^{\ast}\rangle |,
\end{eqnarray*} 
where $p,q>1$ such that $\frac{1}{p}+\frac{1}{q}=1$. Denote
$$
f_{k,j}^{\pm}(z)=t^{-\frac{\theta}{j}}2^{-\frac{k\theta}{j}}e^{\mp i\frac{\psi_{A}}{j}}h_{j}^{\pm}(zt^{-1}2^{-k}),
$$
where 
$$
h_{j}^{\pm}(w)=\big((-w)^{\phi}(-w+e^{\mp i\psi_{A}})^{-1}\big)^{\frac{1}{j}} \in H_{0}^{\infty}(\theta_{A}-\frac{\varepsilon}{2}), \quad j\in\{p,q\}.
$$
Then, Lemma \ref{com} implies
\begin{eqnarray*}
\lefteqn{|\langle W_{m}^{\pm}(t)x,x^{\ast}\rangle|}\\
&=&|\sum_{k=0}^{m-1}\langle f_{k,p}^{\pm}(-A)(B_{c}-t2^{k}e^{\mp i\psi_{A}})^{-1}t2^{k}e^{\mp i\psi_{A}}x,f_{k,q}^{\pm}(-A^{\ast})x^{\ast}\rangle |\\
&\leq&|\sum_{k=0}^{m-1}\langle (B_{c}-t2^{k}e^{\mp i\psi_{A}})^{-1}t2^{k}e^{\mp i\psi_{A}} f_{k,p}^{\pm}(-A)x,f_{k,q}^{\pm}(-A^{\ast})x^{\ast}\rangle |\\
&&+|\sum_{k=0}^{m-1}\langle Q_{f_{k,p}^{\pm}}(c-t2^{k}e^{\mp i\psi_{A}})t2^{k}e^{\mp i\psi_{A}}x,f_{k,q}^{\pm}(-A^{\ast})x^{\ast}\rangle |.
\end{eqnarray*} 
Due to Lemma \ref{easy}, $B_{c}$ belongs again to $\mathcal{ON}(\theta_{B})$ and its $ON$-sectorial bound is uniformly bounded in $c$. Hence, by Lemma \ref{weakon} we obtain 
\begin{eqnarray}\nonumber
\lefteqn{|\langle W_{m}^{\pm}(t)x,x^{\ast}\rangle|}\\\nonumber
&\leq&C_{0}\big(\sup_{a_{k}\in \mathbb{D}}\|\sum_{k=0}^{m-1}a_{k} f_{k,p}^{\pm}(-A)x\|_{E}\big)\big(\sup_{b_{k}\in \mathbb{D}}\|\sum_{k=0}^{m-1} b_{k}f_{k,q}^{\pm}(-A^{\ast})x^{\ast}\|_{E^{\ast}}\big)\\\label{haw}
&&+C_{0}\big(\sup_{a_{k}\in \mathbb{D}}\|\sum_{k=0}^{m-1} a_{k}Q_{f_{k,p}^{\pm}}(c-t2^{k}e^{\mp i\psi_{A}})t2^{k}e^{\mp i\psi_{A}}x\|_{E}\big)\big(\sup_{b_{k}\in \mathbb{D}}\| \sum_{k=0}^{m-1} b_{k}f_{k,q}^{\pm}(-A^{\ast})x^{\ast}\|_{E^{\ast}}\big),
\end{eqnarray} 
for some constant $C_{0}>0$ that depends only on the $ON$-sectorial bound of $B$. 

Concerning the family of bounded operators $Q_{f_{k,p}^{\pm}}(c-t2^{k}e^{\mp i\psi_{A}})t2^{k}e^{\mp i\psi_{A}}$, by Condition \ref{c3} we estimate
\begin{eqnarray*}
\lefteqn{\|Q_{f_{k,p}^{\pm}}(c-t2^{k}e^{\mp i\psi_{A}})t2^{k}e^{\mp i\psi_{A}}\|_{\mathcal{L}(E)}}\\
&\leq& \frac{C}{2\pi}\int_{\Gamma_{\theta_{A}}}\frac{t2^{k}|f_{k,p}^{\pm}(z)|}{(1+|z|^{\alpha_{1}})(1+|c-t2^{k}e^{\mp i\psi_{A}}|^{\beta_{1}})}dz\\
&\leq&\frac{C}{2\pi} t^{-\frac{\theta}{p}}2^{-\frac{k\theta}{p}}\int_{\Gamma_{\theta_{A}}} \frac{t2^{k}|zt^{-1}2^{-k}|^{\frac{\phi}{p}}}{(1+|z|^{\alpha_{1}})(1+|c-t2^{k}e^{\mp i\psi_{A}}|^{\beta_{1}})|zt^{-1}2^{-k}-e^{\mp i\psi_{A}}|^{\frac{1}{p}}}dz.
\end{eqnarray*}
By changing variables and taking appropriate values for $p$ and $\phi$, the last integral in the above inequality converges absolutely and it is uniformly bounded in $t$ and $\theta$ by $2^{(2-\alpha_{1}-\beta_{1})k}$, for each $k$. More precisely, by possibly increasing $C$ we can assume that $\alpha_{1}<2$, and then, by taking $p$ close to $1$ and $\phi$ close to $0$ when $\alpha_{1}<1$ and $p$, $\phi$ both close to $1$ when $\alpha_{1}\geq1$, we have that 
\begin{gather*}
\|Q_{f_{k,p}^{\pm}}(c-t2^{k}e^{\mp i\psi_{A}})t2^{k}e^{\mp i\psi_{A}}\|_{\mathcal{L}(E)}\leq C_{1}2^{(2-\alpha_{1}-\beta_{1})k}
\end{gather*}
for some constant $C_{1}$ independent of $t$, $\theta$ and $k$. Therefore, \eqref{haw} and Lemma \ref{LKW} imply that
\begin{gather*}
|\langle W_{m}^{\pm}(t)x,x^{\ast}\rangle|\leq C_{2} \|x\|_{E}\|x^{\ast}\|_{E^{\ast}},
\end{gather*} 
with some constant $C_{2}$ independent of $m$, $t$ and $\theta$. Hence, by \eqref{gm} we have
\begin{gather*}
\|G_{m}^{\pm}x\|_{E}\leq \frac{C_{2}}{2\pi} \|x\|_{E}
\end{gather*} 
and by taking the limit as $m\rightarrow\infty$ we obtain
\begin{gather}\label{iuy}
\|G_{\theta}x\|_{E}\leq \frac{C_{2}}{\pi} \|x\|_{E}.
\end{gather} 
Clearly,
\begin{gather}\label{ppo}
\|U_{\theta}x\|_{E}\leq C_{3} \|x\|_{E},
\end{gather}
for some constant $C_{3}$ independent of $\theta$. Therefore, \eqref{ppl}, \eqref{iuy} and \eqref{ppo} imply that
\begin{gather*}
\|AA^{-\theta}K_{c}x\|_{E}\leq C_{4} \|x\|_{E},
\end{gather*}
with some constant $C_{4}$ independent of $\theta$. By Lemma \ref{l32}, if $y\in \mathcal{D}(B)$, we have that $K_{c}y\in\mathcal{D}(A)$. Hence, 
\begin{gather*}
\|A^{-\theta}AK_{c}y\|_{E}\leq C_{4} \|y\|_{E},
\end{gather*}
and by taking the limit as $\theta\rightarrow0$ we obtain 
\begin{gather}\label{AK}
\|AK_{c}y\|_{E}\leq C_{4} \|y\|_{E}.
\end{gather}
By the closedness of $A$ and a density argument we conclude that $K_{c}$ maps $E$ to $\mathcal{D}(A)$ and $AK_{c}\in\mathcal{L}(E)$. 

By taking $|c|$ sufficiently large, from Proposition \ref{t1} and Proposition \ref{t12} we have that
\begin{gather}\label{kh8}
A^{-1}(I+P_{c})^{-1}AK_{c}(A+B_{c})=I \quad \text{in} \quad \mathcal{D}(A+B)
\end{gather}
and
\begin{gather}\label{lat}
(A+B_{c})L_{c}=I+T_{c} \quad\, \text{in} \quad \mathcal{D}(A).
\end{gather}
By combining the above equations we find that
\begin{gather}\label{tre}
A^{-1}(I+P_{c})^{-1}AK_{c}(I+T_{c})=L_{c} \quad\, \text{in} \quad \mathcal{D}(A).
\end{gather}
Due to the boundedness of $AK_{c}$ the above formula also holds in $E$. Therefore, we conclude that $L_{c}$ maps $E$ to $\mathcal{D}(A)$ and $AL_{c}\in\mathcal{L}(E)$. Then, by \eqref{lat}, i.e. by
$$
B_{c}L_{c}=I+T_{c}-AL_{c} \quad\, \text{in} \quad \mathcal{D}(A),
$$
the closedness of $B$ and a density argument, we also deduce that $L_{c}$ maps $E$ to $\mathcal{D}(B)$ and $B_{c}L_{c}\in\mathcal{L}(E)$. The right inverse of $A+B_{c}$ then follows by \eqref{lat} and the invertibility of $I+T_{c}$. Hence, \eqref{lat} and \eqref{tre} can be improved to
\begin{gather*}\label{kh7}
(A+B_{c})L_{c}=I+T_{c} \quad\, \text{in} \quad E
\end{gather*}
and 
\begin{gather*}\label{kh88}
A^{-1}(I+P_{c})^{-1}AK_{c}(I+T_{c})=L_{c} \quad \text{in} \quad E.
\end{gather*}
By the invertibility of $I+T_{c}$, the last equation implies that the left inverse of $A+B_{c}$, which is given by \eqref{kh8}, maps to $\mathcal{D}(A+B)$, and therefore closedness follows.

Concerning the sectoriality of $A+B+c_{0}$, for sufficiently large $c_{0}\geq0$, we first note that by Proposition \ref{t12} the norm $\|(I+T_{c})^{-1}\|_{\mathcal{L}(E)}$ is uniformly bounded in $c$ when $|c|\in [c_{0},\infty)$. By changing $z=(1+|c|)\mu$ in the integral formula for $L_{c}$ we find that 
\begin{eqnarray*}
\lefteqn{L_{c}=\frac{1}{2\pi i}\int_{\Gamma_{\theta_{B}}}(B+c+(1+|c|)\mu)^{-1}(A-(1+|c|)\mu)^{-1}(1+|c|)d\mu}\\
&=&\frac{1}{2\pi i}\int_{\Gamma_{\rho,\theta_{B}}}(B+c+(1+|c|)\mu)^{-1}(A-(1+|c|)\mu)^{-1}(1+|c|)d\mu,
\end{eqnarray*}
for some sufficiently small $\rho>0$ independent of $c$. By standard sectoriality we estimate
\begin{gather}\label{secbound}
\|L_{c}\|_{\mathcal{L}(E)}\leq \frac{1}{2\pi(1+|c|)}\int_{\Gamma_{\rho,\theta_{B}}}\frac{\kappa_{A}\kappa_{B}}{((1+|c|)^{-1}+|c(1+|c|)^{-1}+\mu|)((1+|c|)^{-1}+|\mu|)}d\mu,
\end{gather}
where $\kappa_{A}$ and $\kappa_{B}$ are sectorial bounds for $A$ and $B$ respectively. If $c=c_{0}+\nu$, $\nu\in S_{\psi}$, then the above absolutely convergent integral is uniformly bounded in $\nu$. Therefore by 
\begin{gather}\label{inversee}
(A+B_{c})^{-1}=L_{c}(I+T_{c})^{-1}
\end{gather}
we conclude that $\sigma(A+B+c_{0})\in S_{\max\{\theta_{A},\theta_{B}\}}$ and for any $\psi\in[0,\pi-\max\{\theta_{A},\theta_{B}\})$ we have that $A+B+c_{0}\in\mathcal{P}(\psi)$.

Now if $\omega\in[0,\min\{\theta_{A},\theta_{B}\})$ and $s\in S_{\omega}$ we have that $B+s\in\mathcal{P}(\theta_{B})$ when $\theta_{B}\leq\frac{\pi}{2}$ and $B+s\in \mathcal{P}(\pi-\omega)$ when $\theta_{B}>\frac{\pi}{2}$. Due to this observation we can replace $B$ in the previous case by $B+s$ and $c$ by $c_{0}\geq0$ sufficiently large so that $(A+B+c_{0}+s)^{-1}\in\mathcal{L}(E,\mathcal{D}(A)\cap\mathcal{D}(B))$ exists and is given by \eqref{inversee} with $c$ replaced by $c_{0}+s$. The uniform boundedness in $s\in S_{\omega}$ of $\|(I+T_{c_{0}+s})^{-1}\|_{\mathcal{L}(E)}$ and $(1+|s|)\|L_{c_{0}+s}\|_{\mathcal{L}(E)}$ follows respectively by \eqref{Tc} and \eqref{secbound} (by possibly replacing $\theta_{B}$ with $\pi-\omega$). Therefore, by \eqref{inversee} with $c$ replaced by $c_{0}+s$ we deduce that $(1+|s|)\|(A+B+c_{0}+s)^{-1}\|_{\mathcal{L}(E)}$ is also uniformly bounded in $s\in S_{\omega}$.
\end{proof}

By following the proof of the above theorem we make the following observation.

\begin{lemma}\label{l7}
Suppose that in the assumptions of Theorem \ref{t2} we keep $\theta_{B}$ fixed and replace $B$ by a family of operators $B(\xi)\in \mathcal{ON}(\theta_{B})$, $\xi\in \Xi$, indexed by a set $\Xi$, such that $A$ and $B(\xi)$ are resolvent commuting for each $\xi$. Then, 
\begin{itemize}
\item[(i)] The shift $c_{0}$ can be chosen to be equal to zero for each $\xi\in \Xi$. 
\item[(ii)] If the sectorial bound and the $ON$-sectorial bound of $B(\xi)$ are uniformly bounded in $\xi\in \Xi$, then the sectorial bound of $A+B(\xi)\in\mathcal{P}(0)$ can be chosen to be uniformly bounded in $\xi\in\Xi$. Furthermore, the $\mathcal{L}(E)$-norm of $B(\xi)(A+B(\xi)+\nu)^{-1}$ is uniformly bounded in $(\xi,\nu)\in \Xi\times [0,\infty)$.
\end{itemize}
\end{lemma}
\begin{proof}
In the proof of Theorem \ref{t2}, $c_{0}$ was taken large enough in order to make sufficiently small the $\mathcal{L}(E)$-norms of the operators $P_c$ and $T_c$ from Proposition \ref{t1} and Proposition \ref{t12} respectively. Since $A$ and $B(\xi)$ are resolvent commuting for each $\xi$, by \cite[III.4.9.1 (ii)]{Am}, \eqref{Pc} and \eqref{Tc} we clearly have that $P_{c}=T_{c}=0$. Therefore, we can take $c_{0}=0$ for all $\xi$ (see also \cite[Theorem 3.7]{DG}).

Furthermore, only the sectorial bounds of the two summands contribute to the estimate \eqref{secbound}, where $\rho$ can now be chosen to be fixed due to Remark \ref{r6}. Hence, the sectorial bound of $A+B(\xi)$ can be chosen to be uniformly bounded in $\xi\in\Xi$. Finally, the $\mathcal{L}(E)$-norm of $B(\xi)(A+B(\xi)+\nu)^{-1}$ can be estimated by the sectorial bound of $A+B(\xi)$ and the $\mathcal{L}(E)$-norm of $A(A+B(\xi)+\nu)^{-1}$. By \eqref{AK}, $\nu$ does not contribute to the estimate of the last norm and $\xi$ contributes only by the sectorial and the $ON$-sectorial bound of $B(\xi)$. Here we have noted that the operator $Q_{f}$ from Lemma \ref{com} is zero in our case. Thus, the $\mathcal{L}(E)$-norm of $B(\xi)(A+B(\xi)+\nu)^{-1}$ is uniformly bounded in $(\xi,\nu)\in \Xi\times [0,\infty)$.
\end{proof}

\section{An application to the abstract linear non-autonomous parabolic problem}

In this section, we apply the previous result on the closedness and invertibility for the sum of two closed operators in order to recover a classical result on the existence, uniqueness and maximal $L^{p}$-regularity for solutions of the abstract linear non-autonomous parabolic equation. We will require the natural extensions of our operators from the original space to the Bochner space to be $ON$-sectorial. Therefore, we restrict to an ideal subclass of $\mathcal{E}$-sectorial operators, namely to the $R$-sectorial operators. 

\begin{definition}\label{Defrs}
Let $E$ be a complex Banach space, $\theta\in[0,\pi)$ and $\{\epsilon_{k}\}_{k\in\mathbb{N}}$ be the sequence of the Rademacher functions. Denote by $\mathcal{R}_{\kappa}(\theta)$, $\kappa\geq1$, the class of all operators $A\in\mathcal{P}(\theta)$ in $E$ such that for any choice of $\lambda_{1},...,\lambda_{n}\in S_{\theta}\backslash\{0\}$ and $x_{1},...,x_{n}\in E$, $n\in\mathbb{N}$, we have
$$
\big\|\sum_{k=1}^{n}\epsilon_{k}\lambda_{k}(A+\lambda_{k})^{-1}x_{k}\big\|_{L^{2}(0,1;E)} \leq \kappa \big\|\sum_{k=1}^{n}\epsilon_{k}x_{k}\big\|_{L^{2}(0,1;E)}.
$$
The elements in $\mathcal{R}(\theta)=\cup_{\kappa\geq1}\mathcal{R}_{\kappa}(\theta)$ are called {\em $R$-sectorial operators of angle $\theta$}. The constant $\inf\{\kappa\, |\, A\in \mathcal{R}_{\kappa}(\theta)\}$ is called {\em $R$-sectorial bound of $A$} and depends on $A$ and $\theta$. 
\end{definition} 

Let $T>0$, $E_1\overset{d}{\hookrightarrow} E_0$ be a densely and continuously injected complex Banach couple and $A(\cdot)\in C([0,T];\mathcal{L}(E_{1},E_{0}))$ be a continuous family of linear operators. Consider the Cauchy problem
\begin{eqnarray}\label{ANP1}
u'(t)+A(t)u(t)&=&g(t),\quad t\in(0,T],\\\label{ANP2}
u(0)&=&0,
\end{eqnarray}
where $g\in L^{p}(0,T;E_{0})$ for some $p\in(1,\infty)$. We will combine Theorem \ref{t2} with a freezing-of-coefficients type argument (see e.g. \cite[Theorem 5.7]{DHP}) in order to show well-posedness for the above problem. Hence, in the case of UMD spaces we wish to recover \cite[Theorem 2.5]{PS1} or equivalently \cite[Theorem 2.7]{Ar} for the case of continuously dependent over the non-autonomous parameter family $A(\cdot)$. For a different approach to the problem we also refer to \cite{Am2}. 

\begin{theorem}\label{tf}
Assume that $E_{0}$ is $UMD$ and that there exists some $\theta>\frac{\pi}{2}$ such that for each $t\in[0,T]$, $A(t)$ is $R$-sectorial of angle $\theta$. Then, the problem \eqref{ANP1}-\eqref{ANP2} is well-posed, i.e. for any $g\in L^{p}(0,T;E_{0})$ there exists a unique solution $u\in W^{1,p}(0,T;E_{0})\cap L^{p}(0,T;E_{1})$ that depends continuously on $g$.
\end{theorem}
\begin{proof}

Let the Banach spaces $X_{0}=L^{p}(0,T;E_{0})$, $X_{1}=L^{p}(0,T;E_{1})$ and
 $$
 X_{2}=\{u\in W^{1,p}(0,T;E_{0})\, |\, u(0)=0\}.
 $$ 
 Consider the operators $A$ and $B$ in $X_{0}$ defined by 
$$
A:u(t)\mapsto (Au)(t)=A(t)u(t) \quad \text{with}\quad \mathcal{D}(A)=X_{1}
$$ 
and 
$$
B:u(t)\mapsto \partial_{t}u(t)=u'(t) \quad \text{with} \quad \mathcal{D}(B)=X_{2}.
$$ 

For any fixed $\xi\in[0,T]$, the operator $A(\xi):E_{1}\rightarrow E_{0}$ is $R$-sectorial of angle $\theta$. Furthermore, the sectorial bound and the $R$-sectorial bound of $A(\xi)$ can be chosen to be uniformly bounded in $\xi\in [0,T]$. This is easy to see by the continuity of $A(\cdot)$. For convenience, e.g. for the $R$-sectorial bound we argue by contradiction as follows. For each $\xi$ let $C_{\xi}$ be the $R$-sectorial bound of $A(\xi)$. Let $\{\xi_{k}\}_{k\in\mathbb{N}}$ be a sequence in $[0,T]$ such that $C_{\xi_{k}}\rightarrow\infty$ when $k\rightarrow\infty$. By possibly passing to a subsequence we may assume that $\xi_{k}\rightarrow \tilde{\xi}$ as $k\rightarrow\infty$ for certain $\tilde{\xi}\in [0,T]$. Let $\varepsilon>0$ be sufficiently small such that $t\in[\tilde{\xi}-\varepsilon,\tilde{\xi}+\varepsilon]$ (with $[\tilde{\xi}-\varepsilon,\tilde{\xi}+\varepsilon]$ replaced by $[0,\tilde{\xi}+\varepsilon]$ when $\tilde{\xi}=0$ and similarly when $\tilde{\xi}=T$) implies that 
$$
\|A(\tilde{\xi})-A(t)\|_{\mathcal{L}(E_{1},E_{0})}<\frac{1}{2\|A(\tilde{\xi})^{-1}\|_{\mathcal{L}(E_{0},E_{1})}}\min\{\frac{1}{1+C_{\tilde{\xi}}},\frac{1}{1+M_{\tilde{\xi}}}\},
$$ 
where $M_{\tilde{\xi}}$ is a sectorial bound for $A(\tilde{\xi})$. Then, due to Proposition \ref{perp} for such $t$ the operator $A(t)$ is $R$-sectorial and its $R$-sectorial bound is uniformly bounded by $2 C_{\tilde{\xi}}$, which gives us a contradiction. Therefore, each extension of $A(\xi)$ to an operator from $X_{1}$ to $X_{0}$ given by $(A(\xi)u)(t)=A(\xi)u(t)$, which we denote again by $A(\xi)$, is also $R$-sectorial of angle $\theta$ and its sectorial and $R$-sectorial bounds can be chosen to be uniformly bounded in $\xi$. 

Since $E_{0}$ is UMD by \cite[Theorem 8.5.8]{Ha} the operator $B$ admits a bounded $H^\infty$-calculus of angle $\omega$, for any $\omega<\frac{\pi}{2}$. For any $\xi\in[0,T]$ the operators $A(\xi)$ and $B$ are resolvent commuting and hence by Theorem \ref{t2}, for each $\xi$ there exists a $c_{0}(\xi)\geq0$ such that $A(\xi)+B+c_{0}(\xi)$ with $\mathcal{D}(A(\xi)+B+c_{0}(\xi))=X_{1}\cap X_{2}$ in $X_{0}$ is closed and belongs to the class $\mathcal{P}(0)$. Furthermore, by Lemma \ref{l7} each $c_{0}(\xi)$ can be chosen to be equal to zero and the $\mathcal{L}(X_{0})$-norm of $A(\xi)(A(\xi)+B+c)^{-1}$ is uniformly bounded in $(\xi,c)\in[0,T]\times[0,\infty)$.

Take $t_{1},...,t_{n}\in[0,T]$, $n\in\mathbb{N}$, $r>0$ and let $\{\chi_{i}\}_{i\in\{1,...,n\}}$ be a collection of smooth non-negative functions on $\mathbb{R}$ bounded by one such that $\chi_{i}=1$ in $[t_{i}-r,t_{i}+r]$ and $\chi_{i}=0$ outside of $[t_{i}-2r,t_{i}+2r]$ for each $i$. Let $\{\psi_{i}\}_{i\in\{1,...,n\}}$ and $\{\phi_{i}\}_{i\in\{1,...,n\}}$ be two further collections of smooth non-negative functions on $\mathbb{R}$ such that $\chi_{i}=1$ on $\mathrm{supp}(\psi_{i})$ and $\psi_{i}=1$ on $\mathrm{supp}(\phi_{i})$ for each $i$. Choose $\{t_{1},...,t_{n}\}$, $n$ and $r$ in such a way that $\{\phi_{i}\}_{i\in\{1,...,n\}}$ is a partition of unity in $[0,T]$. By the argument in the previous paragraph, each $A(t_{i})+B$ belongs to $\mathcal{P}(0)$. 

For each $i\in\{1,...,n\}$ let 
$$
A_{i}=\chi_{i}(A+B)+(1-\chi_{i})(A(t_{i})+B)=A(t_{i})+B+\chi_{i}(A-A(t_{i}))
$$
with $\mathcal{D}(A_{i})=X_{1}\cap X_{2}$ in $X_{0}$. By taking $r$ sufficiently small and $n$ large enough, from the continuity of $A(\cdot)$ and the uniform boundedness in $(\xi,c)\in[0,T]\times[0,\infty)$ of the $\mathcal{L}(X_{0})$-norm of $A(\xi)(A(\xi)+B+c)^{-1}$, we can achieve $A_{i}\in\mathcal{P}(0)$ for each $i$. More precisely, we have that 
\begin{gather}
(A_{i}+c)^{-1}=(A(t_{i})+B+c)^{-1}\sum_{k=0}^{\infty}\big(\chi_{i}(A(t_{i})-A)(A(t_{i})+B+c)^{-1}\big)^{k}, \quad c\geq0,
\end{gather}
provided that $\|\chi_{i}(A(t_{i})-A)(A(t_{i})+B+c)^{-1}\|_{\mathcal{L}(X_{0})}\leq\frac{1}{2}$. Here we have used the fact that the norm $\|A^{-1}(\xi)\|_{\mathcal{L}(X_{0},X_{1})}$ is uniformly bounded in $\xi\in[0,T]$ due to the continuity of $A(\cdot)$.

Take $u\in X_{1}\cap X_{2}$ and $g\in X_{0}$. By multiplying 
$$
(A+B+c)u=g
$$
with $\phi_{i}$ we get 
$$
(A+B+c)\phi_{i}u=\phi_{i}g+[A+B+c,\phi_{i}]u=\phi_{i}g+\phi'_{i}u,
$$
where we have used the fact that the commutator $[A+B+c,\phi_{i}]$ acts in $X_{1}\cap X_{2}$ by multiplication with $\phi'_{i}$. By applying the inverse of $A_{i}+c$ we obtain 
$$
\phi_{i}u=(A_{i}+c)^{-1}(\phi_{i}g+\phi'_{i}u),
$$
and therefore 
$$
\phi_{i}u=\psi_{i}(A_{i}+c)^{-1}(\phi_{i}g+\phi'_{i}u).
$$
By summing up we find that 
\begin{gather}\label{LI}
u=\sum_{i=1}^{n}\psi_{i}(A_{i}+c)^{-1}\phi_{i}g+\sum_{i=1}^{n}\psi_{i}(A_{i}+c)^{-1}\phi'_{i}u.
\end{gather}
Due to the sectoriality of $A_{i}$, the $\mathcal{L}(X_{1}\cap X_{2})$-norm of $(A_{i}+c)^{-1}$ decays like $c^{-1}$ when $c\rightarrow\infty$. Furthermore, multiplication by $\phi'_{i}$ induces a bounded map in $X_{1}\cap X_{2}$. Hence, by taking $c$ sufficiently large, from \eqref{LI} we obtain a left inverse $L$ of $A+B+c$ that belongs to $\mathcal{L}(X_{0},X_{1}\cap X_{2})$ and is expressed by
\begin{gather}
L=\Big(I-\sum_{i=1}^{n}\psi_{i}(A_{i}+c)^{-1}\phi'_{i}\Big)^{-1}\Big(\sum_{i=1}^{n}\psi_{i}(A_{i}+c)^{-1}\phi_{i}\Big).
\end{gather}

Moreover, from \eqref{LI} we estimate
\begin{eqnarray}\nonumber
\lefteqn{(A+B+c)L=(A+B+c)\sum_{i=1}^{n}\psi_{i}(A_{i}+c)^{-1}(\phi_{i}+\phi'_{i}L)}\\\nonumber
&=&\sum_{i=1}^{n}\psi_{i}(A+B+c)(A_{i}+c)^{-1}(\phi_{i}+\phi'_{i}L)+\sum_{i=1}^{n}[A+B+c,\psi_{i}](A_{i}+c)^{-1}(\phi_{i}+\phi'_{i}L)\\\label{RI}
&=&\sum_{i=1}^{n}\psi_{i}\phi_{i}+\sum_{i=1}^{n}\psi_{i}\phi'_{i}L+\sum_{i=1}^{n}\psi'_{i}(A_{i}+c)^{-1}(\phi_{i}+\phi'_{i}L).
\end{eqnarray}
Note that $\sum_{i=1}^{n}\psi_{i}\phi_{i}=1$ and $\sum_{i=1}^{n}\psi_{i}\phi'_{i}=0$. Also, by the sectoriality of $A_{i}$, the $\mathcal{L}(X_{0})$-norm of $(A_{i}+c)^{-1}$ tends to zero as $c\rightarrow\infty$. Therefore, by possibly increasing $c$, \eqref{RI} provides us a right inverse of $A+B+c$ which belongs to $\mathcal{L}(X_{0},X_{1}\cap X_{2})$. The result now follows by replacing $u(t)$ in \eqref{ANP1} with $e^{ct}v(t)$.
\end{proof}

\end{document}